\documentclass[12pt]{amsart}
\usepackage{graphics}
\usepackage{amsfonts,amssymb,color}
\usepackage[mathscr]{eucal}
\usepackage{amsmath, amsthm}
\usepackage{mathrsfs}
\usepackage{amsbsy}
\usepackage{dsfont}
\usepackage{bbm}
\usepackage{wasysym}
\usepackage{stmaryrd}
\usepackage{url}

\input xypic
\xyoption{all}

\makeindex
\makeglossary

\begin{document}
\baselineskip = 16pt
\newcommand \C{{\mathbb C}}
\newcommand \ZZ {{\mathbb Z}}
\newcommand \NN {{\mathbb N}}
\newcommand \QQ{{\mathbb Q}}
\newcommand \RR {{\mathbb R}}
\newcommand \PR {{\mathbb P}}
\newcommand \AF {{\mathbb A}}
\newcommand \GG {{\mathbb G}}
\newcommand \bcA {{\mathscr A}}
\newcommand \bcC {{\mathscr C}}
\newcommand \bcD {{\mathscr D}}
\newcommand \bcF {{\mathscr F}}
\newcommand \bcG {{\mathscr G}}
\newcommand \bcH {{\mathscr H}}
\newcommand \bcM {{\mathscr M}}
\newcommand \bcJ {{\mathscr J}}
\newcommand \bcL {{\mathscr L}}
\newcommand \bcO {{\mathscr O}}
\newcommand \bcP {{\mathscr P}}
\newcommand \bcQ {{\mathscr Q}}
\newcommand \bcR {{\mathscr R}}
\newcommand \bcS {{\mathscr S}}
\newcommand \bcV {{\mathscr V}}
\newcommand \bcW {{\mathscr W}}
\newcommand \bcX {{\mathscr X}}
\newcommand \bcY {{\mathscr Y}}
\newcommand \bcZ {{\mathscr Z}}
\newcommand \goa {{\mathfrak a}}
\newcommand \gob {{\mathfrak b}}
\newcommand \goc {{\mathfrak c}}
\newcommand \gom {{\mathfrak m}}
\newcommand \gon {{\mathfrak n}}
\newcommand \gop {{\mathfrak p}}
\newcommand \goq {{\mathfrak q}}
\newcommand \goQ {{\mathfrak Q}}
\newcommand \goP {{\mathfrak P}}
\newcommand \goM {{\mathfrak M}}
\newcommand \goN {{\mathfrak N}}
\newcommand \uno {{\mathbbm 1}}
\newcommand \Le {{\mathbbm L}}
\newcommand \Spec {{\rm {Spec}}}
\newcommand \Gr {{\rm {Gr}}}
\newcommand \Pic {{\rm {Pic}}}
\newcommand \Jac {{{J}}}
\newcommand \Alb {{\rm {Alb}}}
\newcommand \Corr {{Corr}}
\newcommand \Chow {{\mathscr C}}
\newcommand \Sym {{\rm {Sym}}}
\newcommand \Prym {{\rm {Prym}}}
\newcommand \cha {{\rm {char}}}
\newcommand \eff {{\rm {eff}}}
\newcommand \tr {{\rm {tr}}}
\newcommand \Tr {{\rm {Tr}}}
\newcommand \pr {{\rm {pr}}}
\newcommand \ev {{\it {ev}}}
\newcommand \cl {{\rm {cl}}}
\newcommand \interior {{\rm {Int}}}
\newcommand \sep {{\rm {sep}}}
\newcommand \td {{\rm {tdeg}}}
\newcommand \alg {{\rm {alg}}}
\newcommand \im {{\rm im}}
\newcommand \gr {{\rm {gr}}}
\newcommand \op {{\rm op}}
\newcommand \Hom {{\rm Hom}}
\newcommand \Hilb {{\rm Hilb}}
\newcommand \Sch {{\mathscr S\! }{\it ch}}
\newcommand \cHilb {{\mathscr H\! }{\it ilb}}
\newcommand \cHom {{\mathscr H\! }{\it om}}
\newcommand \colim {{{\rm colim}\, }} 
\newcommand \End {{\rm {End}}}
\newcommand \coker {{\rm {coker}}}
\newcommand \id {{\rm {id}}}
\newcommand \van {{\rm {van}}}
\newcommand \spc {{\rm {sp}}}
\newcommand \Ob {{\rm Ob}}
\newcommand \Aut {{\rm Aut}}
\newcommand \cor {{\rm {cor}}}
\newcommand \Cor {{\it {Corr}}}
\newcommand \res {{\rm {res}}}
\newcommand \red {{\rm{red}}}
\newcommand \Gal {{\rm {Gal}}}
\newcommand \PGL {{\rm {PGL}}}
\newcommand \Bl {{\rm {Bl}}}
\newcommand \Sing {{\rm {Sing}}}
\newcommand \spn {{\rm {span}}}
\newcommand \Nm {{\rm {Nm}}}
\newcommand \inv {{\rm {inv}}}
\newcommand \codim {{\rm {codim}}}
\newcommand \Div{{\rm{Div}}}
\newcommand \sg {{\Sigma }}
\newcommand \DM {{\sf DM}}
\newcommand \Gm {{{\mathbb G}_{\rm m}}}
\newcommand \tame {\rm {tame }}
\newcommand \znak {{\natural }}
\newcommand \lra {\longrightarrow}
\newcommand \hra {\hookrightarrow}
\newcommand \rra {\rightrightarrows}
\newcommand \ord {{\rm {ord}}}
\newcommand \Rat {{\mathscr Rat}}
\newcommand \rd {{\rm {red}}}
\newcommand \bSpec {{\bf {Spec}}}
\newcommand \Proj {{\rm {Proj}}}
\newcommand \pdiv {{\rm {div}}}
\newcommand \CH {{\it {CH}}}
\newcommand \wt {\widetilde }
\newcommand \ac {\acute }
\newcommand \ch {\check }
\newcommand \ol {\overline }
\newcommand \Th {\Theta}
\newcommand \cAb {{\mathscr A\! }{\it b}}

\newenvironment{pf}{\par\noindent{\em Proof}.}{\hfill\framebox(6,6)
\par\medskip}

\newtheorem{theorem}[subsection]{Theorem}
\newtheorem{conjecture}[subsection]{Conjecture}
\newtheorem{proposition}[subsection]{Proposition}
\newtheorem{lemma}[subsection]{Lemma}
\newtheorem{remark}[subsection]{Remark}
\newtheorem{remarks}[subsection]{Remarks}
\newtheorem{definition}[subsection]{Definition}
\newtheorem{corollary}[subsection]{Corollary}
\newtheorem{example}[subsection]{Example}
\newtheorem{examples}[subsection]{examples}

\title{On the kernel of the push-forward homomorphism between Chow groups.}
\author{ Kalyan Banerjee, Jaya NN  Iyer}
\address {Indian Statistical Institute, Bangalore Centre, Bangalore 560059, India}
\address{The Institute of Mathematical Sciences, CIT
Campus, Taramani, Chennai 600113, India}
\email{kalyanb$_{-}$vs@isibang.ac.in}
\email{jniyer@imsc.res.in}

\footnotetext{Mathematics Classification Number: 14C25, 14D05, 14D20,
 14D21}

\begin{abstract}
In this paper, we prove that the kernel of the push-forward homomorphism on $d$-cycles modulo rational equivalence, induced by the closed embedding of an ample divisor linearly equivalent to some multiple of the theta divisor inside the Jacobian variety $J(C)$ is trivial.  Here $C$ is a smooth projective curve of genus $g$.
\end{abstract}

\maketitle

\setcounter{tocdepth}{1}
\tableofcontents

\section{Introduction}
In this paper, we investigate the kernel of the push-forward homomorphism  on Chow groups induced by the closed embedding of a smooth irreducible ample divisor $D$ inside a smooth projective variety $X$, over the field of complex numbers. Assume the dimension $dim(X)=n$ and let $j: D\hookrightarrow X$ be the closed embedding . This question is motivated by the following results and conjecture.
When Chow groups are replaced by the singular homology of a smooth projective variety over $\C$, the (dual of) Lefschetz hyperplane theorem gives an isomorphism of the pushforward map:
$$
j_*:H_k(D,\ZZ) \hookrightarrow H_k(X,\ZZ)
$$
for $k \,<\,n$, and surjectivity when $k=n$..

M. Nori \cite[Conjecture 7.2.5]{Nori} gave improved bounds on the degrees of singular cohomology for the standard Lefschetz restriction maps, and when $D$ is a very general ample divisor of large degree on $X$.
Furthermore, he conjectured the following  on the restriction maps on the rational Chow groups:
\begin{conjecture}\label{Nor}
Suppose $D$ is a very general smooth ample divisor on $X$, of sufficiently large degree. Then the restriction map:
$$
j^*:CH^p(X)\otimes \QQ \rightarrow CH^p(D)\otimes \QQ
$$
is an isomorphism, for $p<n$ and is injective, for $p=n$.
\end{conjecture}

More generally, we have (see \cite[Conjecture 1.5]{Paranjape}):
\begin{conjecture}\label{Par}
Let $D$ be a smooth ample divisor in $X$. Then the restriction map for the inclusion of $D$ in $X$:
$$
CH^p(X)\otimes \QQ \rightarrow CH^p(D)\otimes \QQ
$$
is an isomorphism, for $p\leq \frac{dim{Y}-1}{2}$.
\end{conjecture}

It seems reasonable to pose the following dual of above Chow Lefschetz questions:
\begin{conjecture}
The pushforward map on the rational Chow groups, for a very general ample divisor $D\subset X$ of sufficiently large degree:
$$
j_*:CH_k(D)\otimes \QQ \rightarrow CH_k(X)\otimes \QQ
$$
is injective, whenever $k>0$.
\end{conjecture}

Similarly, we could pose the dual version of Conjecture \ref{Par}.

Our  aim is to verify these conjectures when $D$ is the theta divisor on the Jacobian of a smooth projective curve, and \textit{special} smooth divisors linearly equivalent to a multiple of the theta divisor.

Let $C$ be a smooth projective curve of genus $g$ and let $\Th$ denote a theta divisor inside the Jacobian $J(C)$ of $C$.

Suppose $\pi:\tilde{C}\rightarrow C$ is a ramified finite Galois covering of degree $n$, for $n\geq 1$. Let $G$ denote the Galois group such that $C=\tilde{C}/G$. Then the induced morphism $\pi^*: J(C)\rightarrow J(\tilde{C})$ is injective. Furthermore, for a suitable translate $\Theta_{\tilde{C}}$ of the theta divisor in $J(\tilde{C})$, the restriction on $J(C)$ is a smooth, irreducible, ample divisor $H_C$  which is linearly equivalent to $n\Th$.

Then we show the following.

\begin{theorem}\label{mainT}
Suppose $C$ is a smooth projective curve of genus $g$ and $H_C$ be as mentioned above. Let $j_C$ denote the closed embedding of  $H_C$  inside $J(C)$. Then the kernel of the push-forward homomorphism $j_{C*}:\CH_k(H_C)\otimes {{\mathbb Q}}\rightarrow \CH_k(J(C))\otimes {{\mathbb Q}}$ is trivial, for $k\geq 0$.
\end{theorem}

 Note that $H_C$ is a special ample divisor in the linear system $|n\Theta|$,  since it is restriction of $\Theta$ on $J(\tilde{C})$. It will be interesting to look at the situation when $H_C$ is a general smooth divisor in $|n\Theta|$. However, as pointed out by C. Voisin, we  cannot expect injectivity on $CH_0(H_C)_{{\mathbb Q}}\rightarrow CH_0(J(C))_{{\mathbb Q}}$, when $H_C$ is very general.

The proof utilises localization sequence of higher Chow groups, applied to $G$-fixed subvarieties of the Jacobian of $\tilde{C}$. An application of a theorem of Collino \cite[Theorem 1]{Collino}, which shows the injectivity, for $k$-cycles on inclusions of lower dimensional symmetric product $Sym^m(C)$ of a curve $C$ inside $Sym^n(C)$, for $m\leq n$, gives us the required injectivity. In the final section \S \ref{Collino},  we also extend his theorem for the pushforward map on higher Chow groups of symmetric powers of a curve, and for any of its open subset. This is crucial in proof of Theorem \ref{mainT}.

Instead of rational Chow groups, the  group $A_k(X)$ of  algebraically trivial $k$-cycles on $X$ modulo rational equivalence can be considered. A weaker problem is posed in the following.

 See  \cite[Exercise 1, Chapter 10]{Voisin}. Let $S$ be a smooth, connected, complex, projective, algebraic surface embedded inside some $\PR^N$. Let $C_t$ be a general smooth hyperplane section of $S$ and $j_t$ be the closed embedding of $C_t$ into $S$. Let $H_t$ be the Hodge structure
$$\ker(j_{t*}:H^1(C_t,\ZZ)\to H^3(S,\ZZ))$$
which is induced from the Hodge structure of $H^3(S,\ZZ)$.
and $A_t$ be the abelian variety corresponding to $H_t$ inside $J(C_t)$. Then
the kernel of the push-forward homomorphism $j_{t*}$ from $A_0(C_t)$ to $A_0(S)$ is a countable union of translates of an abelian subvariety $A_{0,t}$ of $A_t$.
For a very general $C_t$, the abelian variety $A_{0,t}$ is either $0$ or $A_t$.

If the albanese map from $A_0(S)$ to $\Alb(S)$ is not an isomorphism, then for a very general $t$, the kernel of the push-forward homomorphism $j_{t*}$ is countable.

In \cite{BG}, the first author and V. Guletskii  extended the problem to even dimensional smooth projective varieties over any algebraically closed, uncountable ground field. For a smooth cubic fourfold in $\PR^5$, and for a very general hyperplane section on it, it is shown that the kernel of the push-forward homomorphism on algebraically trivial algebraic $1$-cycles modulo rational equivalence, induced by the closed embedding of the hyperplane section into the cubic fourfold, is countable.

We consider  the Jacobian variety $J(C)$ of a smooth projective curve $C$ and the associated Kummer variety $K(J(C)):=\frac{J(C)}{<i>}$. Here $i$ is the inverse map on $J(C)$.
We consider the image of a symmetric  theta divisor $\Th$, i.e. $i(\Th)=\Th$.

We show:
\begin{theorem}
Let $C$ be a hyperelliptic curve of genus four and $D$ denotes the image of a symmetric theta-divisor $\Th$ under the natural morphism  $q:J(C)\to K(J(C))$. Let $j'$ denote the closed embedding of $D$ into $K(J(C))$. Then $A^2(D)$ is trivial and hence the kernel of the push-forward homomorphism $j'_*$ from $A^2(D)$ to $A^3(K(J(C)))$ is trivial.
\end{theorem}

{\small \textbf{Acknowledgements:} We thank C. Voisin, R. Sebastien for pointing out an inaccuracy.
The first named author is grateful to Department of Atomic Energy, India for funding this project.}

\textbf{Notations}:
Here $k$ is an uncountable, algebraically closed field and all the varieties are defined over $k$. Denote
$$
CH_d(X)_{\mathbb Q}:= CH_d(X)\otimes {\mathbb Q}.
$$
Here $X$ is a variety of pure dimension $n$, defined over $k$ and $CH_d(X)$ denotes the Chow group of $d$-dimensional cycles modulo rational equivalence.

We denote $A_d(X)$  the group of algebraically trivial $k$- cycles on $X$ modulo rational equivalence. Let
$$
A^d(X):= A_{dim\,X-d}(X),\, \,CH^d(X):=CH_{dim\,X-d}(X).
$$
We write
$$
CH_d(X,s)_{\mathbb Q}:= CH^{dim\,X -d}(X,s) \otimes \mathbb Q
$$
the Bloch's higher Chow groups with ${\mathbb Q}$-coefficients.
When $X$ is a singular variety, we replace above Chow groups by Fulton's operational Chow groups.
This will be essential in
proof of Theorem \ref{prop2}, where we consider the operational Chow groups of theta divisor $\Theta_C$ which is a singular variety.

\section{Kummer variety of a hyperelliptic curve}

In this section we consider a hyperelliptic curve $C$ of genus $4$ and the Kummer variety $K(J(C))$ associated to the Jacobian $J(C)$ of the curve $C$. Let $\Th$ denote a symmetric theta divisor inside $J(C)$ and let $D$ denote the image of $\Th$ inside $K(J(C))$, under the natural morphism from $J(C)$ to $K(J(C))$. We would like to investigate the kernel of the push-forward homomorphism at the level of Chow groups of one cycles, induced by the closed embedding of $D$ in $K(J(C))$.

First we prove the following two propositions which are true for any smooth projective curve of genus $g$. Define the map $\wt{i}$ from $\Pic^{g-1}C$ to itself, given by
  $$
  \wt{i}(\bcO(D))=K_C\otimes \bcO(-D)\;,
  $$
where for a divisor $D$, $\bcO(D)$ denote the line bundle associated to $D$ and $K_C$ is the canonical line bundle on $C$.
Consider a theta characteristic $\tau$ such that $\tau^2=K_C$. Consider the following map
$$
\otimes \tau^{-1}:\Pic^{g-1}C\to J(C)
$$
given by
$$
\bcO(D)\mapsto \bcO(D)\otimes \tau^{-1})\;.
$$

\begin{lemma}
The following square is commutative.

$$
  \diagram
  \Pic^{g-1}C \ar[dd]_-{\otimes \tau^{-1}} \ar[rr]^-{\wt{i}} & & \Pic^{g-1}C \ar[dd]^-{\otimes \tau^{-1}} \\ \\
   J(C)\ar[rr]^-{i} & & J(C)
  \enddiagram
$$
\end{lemma}

\begin{proof}
First observe that $i\circ (\otimes \tau^{-1})$ is $\bcO(-D)\otimes \tau$. On the other hand $\otimes \tau^{-1}\circ\wt{i}(\bcO(D))$ is equal to
$$
K_C\otimes \bcO(-D)\otimes\tau^{-1}
$$
that is nothing but
$$
\bcO(-D)\otimes \tau^2\otimes \tau^{-1}
$$
which is equal to
$$
\bcO(-D)\otimes \tau\;.
$$
So the above diagram is commutative.

\end{proof}

The commutativity of the above diagram  gives us a map from $\Pic^{g-1}C$ to the Kummer variety $K(J(C))$.

Now assume that $C$ is hyperelliptic, and $h:C\to\PR^1$ be the hyperelliptic map. So we have the following commutative triangle

$$
  \xymatrix{
  C \ar[rr]^-{i}
  \ar[ddrr]_-{h} & &
  C \ar[dd]^-{h} \\ \\
  & & \PR^1
  }
$$
where $i$ is the hyperelliptic involution induced by the degree $2$ morphism $h$.
We will use the following.
\begin{theorem}
\label{theorem1}(\cite{Hartshorne},IV,$5.4$)Let $D$ be an effective special divisor on a smooth curve $C$, then
$$
\dim(|D|)\leq \frac{1}{2}\deg(D)\;.
$$
Furthermore equality occurs if and only if $D=0$ or $D=K_C$ or $C$ is hyperelliptic and $D$ is a multiple of the unique $g^1_2$ on $C$.
\end{theorem}
In the above theorem $g^r_d$ denotes a linear system of dimension $r$ and degree $d$. Also by special divisor we mean a divisor $D$ such that the dimension of the   linear system associated to $K_C-D$ is greater than zero.

Let $l$ be the map from $\Sym^{g-1}C$ to $\Pic^{g-1}C$ given by
$$
l(D)=\bcO(D)\;,
$$
where $\bcO(D)$ denote the line bundle associated to the divisor $D$. Consider $i_C:\Sym^{g-1}C\to \Sym^{g-1}C$ given by
$$
i_C(P_1+\cdots+P_n)=i(P_1)+\cdots+i(P_n)\;,
$$
where $i$ is the hyperelliptic involution on $C$.
We show:

\begin{proposition}
\label{prop1}

The following diagram is commutative.

$$
  \diagram
  \Sym^{g-1}C \ar[dd]_-{i_C} \ar[rr]^-{l} & & \Pic^{g-1}C \ar[dd]^-{\wt{i}} \\ \\
  \Sym^{g-1}C \ar[rr]^-{l} & & \Pic^{g-1}C
  \enddiagram
  $$
In other words, the involution $\wt{i}$ lifts on the $(g-1)$-st symmetric power of the curve.
\end{proposition}
\begin{proof}
First, \cite[Proposition 2.3]{Hartshorne} gives:
$$
K_C=h^*K_{\PR^1}+\bcO(B)
$$
where $B$ is the branch divisor of the morphism $h$ and degree of $B$ is $2g+2$. Now we have to show that $\bcO(i_C(D))$ is $K_C-\bcO(D)$. In other words, we have to prove that
$$
\bcO(i_C(D))\otimes \bcO(D)=K_C
$$
that is
$$\bcO(D+i_C(D))=K_C\;.$$
Here $i_C$ is the involution induced on the symmetric powers of $C$ defined above, by the involution $i$ on $C$. Observe that $D+i_C(D)$ is invariant under the involution $i_C$.

Now consider the morphism $h:C\to \PR^1$.
We compute $h^0(K_C-D-i_C D)$, that is the dimension of the vector space of global sections of the line bundle  $K_C-\bcO(D+i_C D)$. By Riemann-Roch theorem we have that
$$
h^0(\bcO(D+i_C D))-h^0(K_C-\bcO(D+i_C D))=2g-2-g+1=g-1\;.
$$
Observe that $\deg(K_C-\bcO(D+i_C D))=0$.
Now for a divisor $D$ the degree is zero means that either the divisor is zero, in this case we have $h^0(D)$ is one or $D$ is non-zero. In the case $D$ is non-zero, we have $h^0(D)=0$, otherwise the line bundle associated to $D$ would be trivial.

So we have two cases

$$
K_C=\bcO(D+i_C D)
$$
or
$$
h^0(K_C-\bcO(D+i_C D))=0\;.
$$

Suppose that $h^0(K_C-\bcO(D+i_C D))=0$.
So by the Riemann-Roch theorem we get that
$$
h^0(\bcO(D+i_C D))=2g-2-g+1=g-1\;.
$$

By the theorem \ref{theorem1} we get that $\bcO(D+iD)$ is equal to $L^{g-1}$ for a line bundle $L\in g^1_2$ on $C$. We have
$$
K_C=h^*K_{\PR^1}+\bcO(B)
$$
and also by \ref{theorem1} we get that any two divisors of degree $2g-2$ on a hyper-elliptic curve $C$, are linearly equivalent, that is the corresponding line bundles on $C$ are isomorphic. This tells us that $h^*\bcO_{\PR^1}(g-1)$ and $L^{g-1}$ are isomorphic. By the projection formula we get that
$$
h_*L^{g-1}=h_*h^*(\bcO_{\PR^1}(g-1))
$$
which is nothing but
$$
\bcO_{\PR^1}(g-1)\oplus \bcO_{\PR^1}(g-1)\otimes \bcO(B)\;.
$$
Since $H^0(C,L^{g-1})$ is isomorphic to to $H^0(\PR^1,h_*L^{g-1})$, we have $h^0(L^{g-1})$ is greater than $g-1$ which is a contradiction. So the possibility that $h^0(K_C-\bcO(D+i_C D))=0$ is ruled out and we have the only  possibility
$$
K_C=\bcO(D+i_C D)\;.
$$
This gives us the commutativity of the diagram.
$$
  \diagram
  \Sym^{g-1}C \ar[dd]_-{i} \ar[rr]^-{l} & & \Pic^{g-1}C \ar[dd]^-{\wt{i}} \\ \\
  \Sym^{g-1}C \ar[rr]^-{l} & & \Pic^{g-1}C
  \enddiagram
  $$
This ends the proof.
\end{proof}
Next, for Chow groups computations, we identify $\Pic^{g-1}C$ with $J(C)$ using a base point $P_0\in C$. The image of $\Sym^{g-1}C$ in $\Pic^{g-1}C$ is denoted by $\Th$ and it is symmetric under $\wt{i}$, by Proposition \ref{prop1}.
\begin{theorem}
Let $C$ be a hyperelliptic curve of  genus $4$ and let $K(\Pic^{3}C)$ denote the Kummer variety associated to $\Pic^{3}C$. Let $D$ denote the image of a symmetric theta-divisor $\Th$ under the natural morphism from $\Pic^{3}C$ to $K(\Pic^{3}C)$. Let $j'$ denote the closed embedding of $D$ into $K(\Pic^{3}C)$. Then $A^2(D)$ is trivial and hence the kernel of the push-forward homomorphism $j'_*$ from $A^2(D)$ to $A^3(K(\Pic^{3}C))$ is trivial.
\end{theorem}
\begin{proof}
The commutativity of the diagram in \ref{prop1} gives us a map from
$\Sym^{3}C/i$ to $\Pic^{3}C/\wt{i}\cong K(\Pic^{3}C)$, where the first morphism is birational and the second one is finite. Now $\Sym^{3}C/i$ is isomorphic to $\Sym^{3}\PR^1$, which is isomorphic to the projective space $\PR^{3}$. Note that $A^2(\PR^{3})$ is trivial hence weakly representable. Since weak representability of $A^2$ is a birational invariant, we get that $A^2(D)$ is isomorphic to an abelian variety $A$. By the proposition $6$ in \cite{BG} we get that the kernel of the push-forward homomorphism $j'_*$ from $A^2(D)$ to $A^3(K(\Pic^{3}C))$ is a countable union of translates of an abelian subvariety $A_0$ of the abelian variety $A$ representing $A^2(D)$. Since $H^{3}(\PR^{3},\ZZ)$ is trivial, we get that the abelian variety $A$ is trivial. So the kernel of the push-forward homomorphism $j'_*$ is trivial.
\end{proof}


\section{Inclusion of theta divisor into the Jacobian}

In this section we investigate the kernel of the push-forward homomorphism, induced by the closed embedding of the theta divisor inside the Jacobian of a smooth projective curve $C$ of genus $g$. More precisely we prove the following theorem.
\begin{theorem}
\label{prop2}
Let $C$ be a smooth projective curve of genus $g$. Let $\Th$ be a symmetric  theta-divisor embedded inside $J(C)$ and let $j$ denote the embedding. Then the kernel of the push-forward homomorphism $j_*$ from ${\CH_d(\Th)}_{\QQ}$ to ${\CH_d(J(C))}_{\QQ}$ is trivial.
\end{theorem}

Since $\Th$ is a singular variety $\CH_d(\Th)$ will denote Fulton's operational Chow groups. In particular, pullback morphisms on these groups are defined induced by arbitrary morphisms $X\rightarrow \Th$. Note that operational Chow groups of a smooth variety are the same as the usual Chow groups.

\begin{proof}
It is well known that the map from $\Sym^{g-1} C$ to $\Th$  is surjective and birational. Let us fix a point $P$ in $C$. Consider the following map $j_C$ from $\Sym^{g-1}C$ to $\Sym^g C$
defined by
$$
P_1+\cdots+ P_{g-1}\mapsto P_1+\cdots+P_{g-1}+P\;.
$$
Here the sum denotes the unordered set of points of lengths $(g-1)$ and $(g)$.

With this definition of $j_C$ we observe that the following diagram is commutative.

$$
  \diagram
  \Sym^{g-1} C \ar[dd]_-{j_C} \ar[rr]^-{q_{\Th}} & & \Th \ar[dd]^-{j} \\ \\
  \Sym^{g} C \ar[rr]^-{q} & & Pic^g(C)
  \enddiagram
  $$

We prove that commutativity of this diagram gives us the following formula at the level of $\CH_d$.
$$
j_*=q_*\circ j_{C*}\circ q_{\Th}^*
$$
To prove this we notice that  for a prime $k$-cycle $V$ in $\CH_d(\Th)$ we have
$$(q\circ j_C)(q_{\Th}^{-1}(V))=(j\circ q_{\Th})(q_{\Th}^{-1}(V))$$
by the commutativity of the above diagram. Now $q_{\Th}$ is surjective, so
$$
q_{\Th}(q_{\Th}^{-1}(V))=V\;.
$$

Now suppose that $E$ is the exceptional locus of $q_{\Th}$ in $\Sym^{g-1} C$ such that $q_{\Th}$ is injective on $\Sym^{g-1} C\setminus E$ into $\Th$. Now suppose $\alpha$ be a cycle class in $\CH_d(\Th)$ and it is non-zero and supported on the complement of $E$, then $q_{\Th}^*$ is non-zero due to birationality. Suppose $\alpha$ is supported on $E$ and $q_{\Th}^*(\alpha)$ is zero. Then we prove that $\alpha$ is torsion.

\begin{lemma}
Let $E$ be a closed subscheme in $\Sym^m C$. Then the closed embedding of $E$ into $\Sym^n C$ for $m\leq n$ induces a push-forward homomorphism at the level of Chow groups which has torsion kernel.
\end{lemma}

\begin{proof}

Let $E$ be a closed subscheme inside $\Sym^m C$, then we consider the embedding of $\Sym^m C$ into $\Sym^n C$. We want to prove that $j:E\to \Sym^n C$ gives rise to an injective push-forward homomorphism at the level of Chow groups. Consider the projection from $\pi_n^{-1}(E)$ to $C^m$. The elements of $\pi_n^{-1}(E)$ are of the form $(x_1,\cdots,p,\cdots,p,\cdots,x_m)$ or $(x_1,\cdots,x_m,p,\cdots,p)$. Therefore the elements of the image of the projection from $\pi_n^{-1}(E)$ will look like $(x_1,p,\cdots,x_j)$ or $(x_1,\cdots,x_m)$. So the image will be a union of disjoint Zariski closed subsets of $C^m$, one of which is $\pi_m^{-1}(E)$. So consider the correspondence $\Gamma'$ given by  the Graph of the projection from $\pi_n^{-1}(E)$ to $\Sym^m C$. Then define $\Gamma$ to be $\pi_n\times \pi_m(\Gamma')$, that will give us a correspondence on $\Sym^n C\times E$. Then by \ref{lemma1} we get that the homomorphism $\Gamma_*j_*$ is induced by the cycle $(j\times id)^*(\Gamma)$. Now we compute this cycle.
So $(j\times id)^{-1}(\Gamma)$ is nothing but
$$\{([e_1,\cdots,e_m],[e_1',\cdots,e_m'])|
([e_1,\cdots,e_m,p\cdots,p],[e_1',\cdots,e_m'])\in \Gamma\}\;.$$
That would mean the following
$(e_1,\cdots,p,\cdots,e_m)$ and $(e_1,\cdots,e_m,p,\cdots,p)$ are in $\pi_n^{-1}(E)$ and
$(e_1',\cdots,e_m')$
is in the image of the projection. So we have
$$(e_1',\cdots,e_m')=(e_1,\cdots,e_m)$$
or
$$e_i'=p$$
for some $i$. That would mean that $(j\times id)^{-1}(\Gamma)=\Delta_E\cup Y$ where $Y$ is supported on $\Sym^{m-1}C\cap E$. Arguing as in \cite{Collino} we get that
$$(j\times id)^*(\Gamma)=d\Delta_E+D$$
where $D$ is supported on $\Sym^{m-1}C\cap E$. Then consider $\rho$ to be the open immersion of the complement of $\Sym^{m-1}C\cap E$ in $E$. Since $D$ is supported on $\Sym^{m-1}C\cap E$ we get that
$$\rho^*\Gamma_*j_*(Z)=\rho^*(dZ)\;.$$
As previous consider the diagram.
$$
  \xymatrix{
   \CH_*(\Sym^{m-1}C\cap E) \ar[r]^-{j'_{*}} \ar[dd]^-{}
  &   \CH_*(E) \ar[r]^-{\rho^{*}} \ar[dd]_-{j_{*}}
  & \CH_*(X_0(m))  \ar[dd]_-{}  \
  \\ \\
    \CH_*(\Sym^{m-1}C\cap E) \ar[r]^-{j''_*}
    & \CH_*(\Sym^{n} C) \ar[r]^-{}
  & \CH^*(U)
  }
$$

$X_0(m),U$ are complement of $\Sym^{m-1}C\cap E$ in $E,\Sym^n C$ respectively. Then suppose that we have
$$j_*(z)=0$$
that gives us that
$$\rho^*\Gamma_*j_*(z)=\rho^*(dz)=0$$
so there exists some $z'$ such that $j_*'(z')=dz$. But by the above diagram we have $j''_*(z')=0$. So by induction if we assume that $j''_*$ has torsion kernel then we get that $d'z'=0$, so we have $dd'z=0$. So the kernel of the map from $\CH_*(E)\to \CH_*(\Sym^n C)$ is torsion, consequently $\CH_*(E)\to \CH_*(\Sym^m C)$ has torsion kernel.
\end{proof}

Since $\Sym^{g}C$ is a blow up of $J(C)$, we have the the inverse image $E'$ over $E$ a projective bundle. By the projective bundle formula we have $\CH_*(E)\to \CH_*(E')$ injective and by the above lemma $\CH_*(E')\to \CH_*(\Sym^{g-1}C)$ is torsion. It will follow from this that if $\alpha$ is supported on $E$ and $q_{\Th}^*(\alpha)$ is zero then $\alpha$ is torsion. So considering Chow groups with $\QQ$-coefficients we get that $q_{\Th}^*$ is injective.

Since $j_{C*}$ is injective by theorem $1$ in \cite{Collino} from ${\CH_d(\Sym^{g-1} C)}_{\QQ}$ to ${\CH_d(\Sym^g C)}_{\QQ}$, we get that $j_{C*}q_{\Th}^*$ is injective. Since $q_{\Th}^*(\alpha)$ is not zero, we get that $j_{C*}\circ q_{\Th}^*(\alpha)$ is not supported on the exceptional locus of $q$. This is because of the fact that $q_{\Th}$ is the restriction of $q$. Now consider the following fiber square,

$$
  \diagram
  \Sym^{g} C \setminus E'\ar[dd]_-{} \ar[rr]^-{} & & \Sym^g C \ar[dd]^-{q} \\ \\
  U \ar[rr]^-{} & & J(C)
  \enddiagram
  $$
where $E'$ is the exceptional locus of the map $\Sym^{g} C\to J(C)$, and $U$ be the open subscheme of $J(C)$ such that $\Sym^g\setminus E'$ is isomorphic to $U$.
This fiber square gives us the following commutative square at the level of Chow groups.
$$
  \diagram
  {\CH_d(\Sym^g C)}_{\QQ} \ar[dd]_-{q_*} \ar[rr]^-{} & & {\CH_d(\Sym^g C \setminus E')}_{\QQ}\ar[dd]^-{} \\ \\
  {\CH_d(J(C))}_{\QQ} \ar[rr]^-{} & & {\CH_d(U)}_{\QQ}
  \enddiagram
  $$
Since $j_*q_{\Th}^*(\alpha)$ is not supported on $E'$, we get that the image of $j_*q_{\Th}^*(\alpha)$ under the pull back homomorphism ${\CH_d(\Sym^g C)}_{\QQ}\to {\CH_d(\Sym^g C\setminus E')}_{\QQ}$ is nonzero. Also observe that the right vertical homomorphism above is an isomorphism, so $j_*q_{\Th}^*(\alpha)$ is mapped to some non-zero element in ${\CH_d(U)}_{\QQ}$. By the commutativity of the above diagram, we get that $q_*(j_*q_{\Th}^*(\alpha))$ is non-zero. In other words $j_*(\alpha)$ is non-zero. So $j_*$ is injective.
\end{proof}

\subsection{Finite group quotients of $J(C)$}

Now we prove that the kernel of the push-forward homomorphism from $\CH_d(D)$ to $\CH_d(K(J(C)))$ is trivial. More generally let $G$ be a finite group acting on $J(C)$, where $C$ is a smooth projective curve of genus $g$. Let $\Th$ denote the theta divisor of $J(C)$ such that $G(\Th)=\Th$. Then we prove the following.

\begin{proposition}
Let $j_G$ denote the embedding of $\Th/G$ into $J(C)/G$. Then the kernel of the push-forward homomorphism $j_{G*}$ from $\CH_d(\Th/G)_{{\mathbb Q}}$ to $\CH_d(J(C)/G)_{{\mathbb Q}}$ is trivial.
\end{proposition}
\begin{proof}
By theorem \ref{prop2}, it suffices to check that the action of $G$ intertwines with $j_*$. That is we have to show that
$$
g.j_*(a)=j_*(g.a)
$$
for any $a$ in $\CH_d(\Th)$ and for any $g\in G$. For that write $a$ as $\sum n_i V_i$. Then
$$
g.(j_*(a))=g.(\sum n_i j(V_i))=\sum n_i g(V_i)
$$
since $j$ is a closed embedding, we have
$$
\sum n_i g(V_i)=j_*(\sum n_i g(V_i))
$$
that is same as
$$
j_*(g.a)\;.
$$
By \cite[Example 1.7.6]{Fulton},  we have
$$
\CH_d(\Th/G)_{{\mathbb Q}}={\CH_d(\Th)^G}_{{\mathbb Q}}
$$
where $\CH_d(\Th)^G_{{\mathbb Q}}$ denotes the $G$-invariants in $\CH_d(\Th)_{{\mathbb Q}}$. By the above intertwining of the group action of $G$, we get that ${j_*}|_{\CH_d(\Th)^G_{{\mathbb Q}}}$ takes it values in $\CH_d(J(C))^G_{{\mathbb Q}}$. Since $j_*$ is injective we get that ${j_*}|_{\CH_d(\Th)^G_{{\mathbb Q}}}$ is injective and
${j_*}|_{\CH_d(\Th)^G_{{\mathbb Q}}}$ is nothing but $j_{G*}$.  So we get that $j_{G*}$ is injective.
\end{proof}

\section{Special ample smooth divisors on $J(C)$}

Let $n\Theta$ denote the $n$-th multiple of $\Th_C$, that is
$$\Theta+\cdots+\Theta$$
$n$ times, inside the Jacobian of a genus $g$ smooth projective curve $C$.
Since $h^0(n\Th_C)=n^g$, we can choose $H_C$, a smooth, irreducible, ample divisor on $J(C)$, linearly equivalent to $n\Th$. We are interested to investigate the kernel of the push-forward homomorphism at the level of Chow groups with rational coefficients, induced by the closed embedding of $H_C$ into $J(C)$.
Consider a Galois covering
$$\pi: \wt{C}\lra C$$
of degree $n$ branched along $r$ points where $r\geq 1$. In particular let $G$ be a finite group acting on $\wt{C}$ such that $C=\wt{C}/G$.

Let $\pi^*$ denote the morphism induced by $\pi$ from $J(C)$ to $J(\wt{C})$. Since $\pi^*$ is injective by \cite[Corollary 11.4.4]{BL}, we identity the image of $\pi^*$ with the polarized pair $(J(C), H_C)$.
 Let us denote genus of $\tilde{C}$ by $\wt{g}$. Note that for a general translate of $\Th_{\wt{C}}$, the restriction of the translate to $J(C)$ is smooth and irreducible. Since by \cite[Lemma 12.3.1]{BL},
$$
(\pi^*)^*(\Th_{\wt{C}})\equiv n\Th_C\equiv H_C\;,
$$
we have $H_C$ is equal to $J(C)\cap \Th_{\wt{C}}$, and it is smooth and irreducible.

Note that $H_C$ is special in the linear system  $|n\Th_C|$ since it is restriction of $\Th_{\tilde{C}}$ and for a general member of $n\Th_C$, this does not happen.

Denote $CH_*(H_C)_{\mathbb Q}:= CH_*(H_C)\otimes {\mathbb Q}$ and $CH_*(J(C))_{\mathbb Q}:= CH_*(J(C))\otimes {\mathbb Q}$. In the following, we identify $Pic^g(C)= J(C)$ and $Pic^{\tilde{g}}(\tilde{C})= J(\tilde{C})$ (without specifying a choice of base point).

\begin{theorem}
Let $C$ be a curve of genus $g$ and $H_C$ be as mentioned above. Let $j_C$ denote the closed embedding of  $H_C$  inside $J(C)$. Then the kernel of the push-forward homomorphism $j_{C*}$ from ${\CH_d(H_C)}_{{\mathbb Q}}$ to $\CH_d(J(C))_{{\mathbb Q}}$ is trivial, for $k\geq 1$.
\end{theorem}

\begin{proof}

By the above discussion we have the following commutative diagram
  $$
  \diagram
   H_C\ar[dd]_-{j_C} \ar[rr]^-{} & & \Th_{\wt{C}} \ar[dd]^-{j_{\wt{C}}} \\ \\
  J(C) \ar[rr]^-{} & & J(\wt{C}).
  \enddiagram
  $$
This diagram gives us the following commutative diagram at the level of $\CH_*$.
  $$
  \diagram
   \CH_d(H_C)_{{\mathbb Q}}\ar[dd]_-{j_{C*}} \ar[rr]^-{} & & \CH_d(\Th_{\wt{C}})_{{\mathbb Q}} \ar[dd]^-{j_{\wt{C}*}} \\ \\
  \CH_d(J(C))_{{\mathbb Q}} \ar[rr]^-{} & & \CH_d(J(\wt{C}))_{{\mathbb Q}}
  \enddiagram
  $$
Using \ref{prop2} we get that $j_{\wt{C}*}$ is injective. To prove that the homomorphism $j_{C*}$ is injective we use the localization exact sequence of Bloch's higher Chow groups \cite{Bloch}. First note that $\Sym^{\wt{g}-1} \wt{C}$ is birational to $\Th_{\wt{C}}$, and $\Sym^{\wt{g}} \wt{C}$ is birational to $J(\wt{C})$. Consider the natural morphism from $\Sym^{\wt{g}}(\wt{C})$ to $J(\wt{C})$. Let $H_C',J(C)'$ denote the scheme theoretic inverse images of $H_C,J(C)$ in $\Sym^{\wt{g}} \wt{C}$. Now fix a base-point $P_0$ in $\wt{C}$ and we consider the inclusion $\Sym^{\wt{g}-1}\wt{C}\hookrightarrow \Sym^{\wt{g}}\wt{C}$ given by
$$
P_1+\cdots+P_{\wt{g}-1}\mapsto P_1+\cdots+P_{\wt{g}-1}+P_0\;.
$$
Then by using the localization exact sequence at the level of higher Chow groups we have the following commutative diagram:

$$
  \xymatrix{
   \CH_d(\Sym^{\wt{g}-1} \wt{C},1)\ar[r]^-{}\ar[dd]_-{}  & \CH_d(\Sym^{\wt{g}-1} \wt{C}\setminus H_C',1) \ar[r]^-{} \ar[dd]_-{}
  &   \CH_d(H_C') \ar[r]^-{} \ar[dd]_-{}
  & \CH_d(\Sym^{\wt{g}-1} \wt{C})\ar[dd]_-{}\
  \\ \\
  \CH_d(\Sym^{\wt{g}} \wt{C},1) \ar[r]^-{} & \CH_d(\Sym^{\wt{g}} \wt{C}\setminus J(C)',1) \ar[r]^-{}
    & \CH_d(J(C)') \ar[r]^-{}
  & \CH_d(\Sym^{\wt{g}} \wt{C})
  }
$$
Since $C=\wt{C}/G$, consider the  induced action of $G$ on symmetric powers of $\wt{C}$.
The group $G$ acts component wise and we have $\Sym^{\wt{g}-1} \wt{C}/G$ is isomorphic to $\Sym^{\wt{g}-1} C$ and similarly $\Sym^{\wt{g}} \wt{C}/G$ is isomorphic to $\Sym^{\wt{g}} C$. Since $G$ acts trivially on $C$, we get that $G$ acts trivially on $J(C)'$ and $H_C'$ respectively. So $G$ acts trivially on $\CH_d(H_C')$ and $\CH_d(J(C)')$. Now consider the $G$-invariant part of the above $G$-equivariant commutative diagram with ${\mathbb Q}$-coefficients. That is consider
$$
  \xymatrix{
   \CH_d(\Sym^{\wt{g}-1} \wt{C},1)^G_{{\mathbb Q}}\ar[r]^-{}\ar[dd]_-{}  & \CH_d(\Sym^{\wt{g}-1} \wt{C}\setminus H_C',1)^G_{{\mathbb Q}} \ar[r]^-{} \ar[dd]_-{}
  &   \CH_d(H_C')^G_{{\mathbb Q}} \ar[r]^-{} \ar[dd]_-{}
  & \CH_d(\Sym^{\wt{g}-1} \wt{C})^G_{{\mathbb Q}}\ar[dd]_-{}\
  \\ \\
  \CH_d(\Sym^{\wt{g}} \wt{C},1)^G_{{\mathbb Q}} \ar[r]^-{} & \CH_d(\Sym^{\wt{g}} \wt{C}\setminus J(C)',1)^G_{{\mathbb Q}} \ar[r]^-{}
    & \CH_d(J(C)')^G_{{\mathbb Q}} \ar[r]^-{}
  & \CH_d(\Sym^{\wt{g}} \wt{C})^G_{{\mathbb Q}}
  }
$$
and it becomes
$$
  \xymatrix{
   \CH_d(\Sym^{\wt{g}-1} C,1)_{{\mathbb Q}}\ar[r]^-{}\ar[dd]_-{}  & \CH_d(\Sym^{\wt{g}-1} C\setminus H_C',1)_{{\mathbb Q}} \ar[r]^-{} \ar[dd]_-{}
  &   \CH_d(H_C')_{{\mathbb Q}} \ar[r]^-{} \ar[dd]_-{}
  & \CH_d(\Sym^{\wt{g}-1} C)_{{\mathbb Q}}\ar[dd]_-{}\
  \\ \\
  \CH_d(\Sym^{\wt{g}} C,1)_{{\mathbb Q}} \ar[r]^-{} & \CH_d(\Sym^{\wt{g}}C \setminus J(C)',1)_{{\mathbb Q}} \ar[r]^-{}
    & \CH_d(J(C)')_{{\mathbb Q}} \ar[r]^-{}
  & \CH_d(\Sym^{\wt{g}} C)_{{\mathbb Q}}
  }
$$
The formula in \cite[Example 1.7.6]{Fulton} also holds for the higher Chow groups.

Since $C$ is of genus $g$ and $\Sym^{\wt{g}-1} C,\Sym^{\wt{g}} C$ are of dimension $\wt{g}-1,\wt{g}$ respectively, we get that $\Sym^{\wt{g}-1} C,\Sym^{\wt{g}} C$ are projective bundles $\PR^{\wt{g}-g-1}_{J(C)},\PR^{\wt{g}-g}_{J(C)}$ respectively, that is a $\PR^{\wt{g}-g-1},\PR^{\wt{g}-g}$ bundle respectively and $\tilde{g}\geq g+1$. So by the projective bundle formula as in \cite{Bloch} we have, when $d=1$,
$$
\CH_1(\PR^{\wt{g}-g-1}_{J(C)},1)_{{\mathbb Q}}=H^{\wt{g}-g-2}.\CH_0(J(C),1)_{{\mathbb Q}}\oplus H^{\wt{g}-g-1}.\CH_1(J(C),1)_{{\mathbb Q}}
$$
and
$$
\CH_1(\PR^{\wt{g}-g}_{J(C)},1)_{{\mathbb Q}}=H^{\wt{g}-g-1}.\CH_0(J(C),1)_{{\mathbb Q}}\oplus H^{\wt{g}-g}.\CH_1(J(C),1)_{{\mathbb Q}}
$$
where $H$ denote the class of the line bundle $\bcO_{\PR_{J(C)}}(1)$ in $\Pic(\PR_{J(C)})$
and we have
$$
\CH_1(\PR^{\wt{g}-g-1}_{J(C)})_{{\mathbb Q}}\,\simeq \,\CH_1(\PR^{\wt{g}-g}_{J(C)})_{{\mathbb Q}}.
$$

Similarly, we could apply projection bundle formula for all $k\geq 0$, to deduce isomorphisms.

 By Corollary \ref{openCollino} in section \ref{Collino}, proved for higher Chow groups with ${\mathbb Q}$-coefficients, we have that the homomorphism
$$
\CH_d(\Sym^{\wt{g}-g-1} C\setminus H_C',1)_{{\mathbb Q}}\rightarrow \CH_d(\Sym^{\wt{g}-g} C\setminus J(C)',1)_{{\mathbb Q}}
$$
 is injective. Also by Collino's theorem in \cite[Theorem 1]{Collino} we have that the homomorphism $$
\CH_d(\Sym^{\wt{g}-g-1}C)_{{\mathbb Q}} \rightarrow \CH_d(\Sym^{\wt{g}-g}C)_{{\mathbb Q}}
$$
 is injective. Now suppose that we start with some non-zero element in $\CH_d(H_C')_{{\mathbb Q}}$, and suppose that it goes to something non-zero in $\CH_d(\Sym^{\wt{g}-g-1}C)_{{\mathbb Q}}$, since the homomorphism from
$\CH_d(\Sym^{\wt{g}-g-1}C)_{{\mathbb Q}}$ to $\CH_d(\Sym^{\wt{g}-g}C)_{{\mathbb Q}}$ is injective, we get that the non-zero element we started with in $\CH_d(H_C')_{{\mathbb Q}}$ goes to some non-zero element in $\CH_d(J(C)')_{{\mathbb Q}}$. Now suppose that the element that we choose from $\CH_d(H_C')_{{\mathbb Q}}$ goes to zero under the homomorphism $\CH_d(H_C')_{{\mathbb Q}}$ to $\CH_d(\Sym^{\wt{g}-g-1}C)_{{\mathbb Q}}$. Then by the localization exact sequence, it follows that the element in $\CH_d(H_C')_{{\mathbb Q}}$ is in the image of the homomorphism
$$
\CH_d(\Sym^{\wt{g}-g-1} C\setminus H_C',1)_{{\mathbb Q}}\to \CH_d(H_C')_{{\mathbb Q}}\;.
$$
Suppose the element in $\CH_d(\Sym^{\wt{g}-g-1} C\setminus H_C',1)_{{\mathbb Q}}$ is nonzero. Since the map
$$
\CH_d(\Sym^{\wt{g}-g-1} C\setminus H_C',1)_{{\mathbb Q}}\rightarrow \CH_d(\Sym^{\wt{g}-g} C\setminus J(C)',1)_{{\mathbb Q}}
$$
 is injective, the image of that element in $\CH_d(\Sym^{\wt{g}-g} C\setminus J(C)',1)_{{\mathbb Q}}$ is non-zero. Either this element goes to zero or it is mapped to a nonzero element in $\CH_d(J(C)')_{{\mathbb Q}}$. If it goes to a nonzero element, then the element we started with from $\CH_d(H_C')_{{\mathbb Q}}$ goes to a non-zero element in $\CH_d(J(C)')_{{\mathbb Q}}$. Suppose the element in $\CH_d(\Sym^{\wt{g}-g} C\setminus J(C)',1)_{{\mathbb Q}}$ goes to zero in $\CH_d(J(C)')_{{\mathbb Q}}$. Then the element is in the image of the map
$$
\CH_d(\Sym^{\wt{g}-g} C,1)_{{\mathbb Q}}\to \CH_d(\Sym^{\wt{g}-g} C\setminus J(C)',1)_{{\mathbb Q}}\;.
$$
Then by using the isomorphism $\CH_d(\Sym^{\wt{g}-g-1} C,1)_{{\mathbb Q}}$ with $\CH_d(\Sym^{\wt{g}-g} C,1)_{{\mathbb Q}}$, we get that this element, comes from an element in $\CH_d(\Sym^{\wt{g}-g-1} C,1)_{{\mathbb Q}}$, then composing the two maps
$$\CH_d(\Sym^{\wt{g}-g-1} C,1)_{{\mathbb Q}}\to \CH_d(\Sym^{\wt{g}-g-1} C\setminus H_C',1)_{{\mathbb Q}}\to \CH_d(H_C')_{{\mathbb Q}}
$$
we get that the element we started with in $\CH_d(H_C')_{{\mathbb Q}}$ is zero, which is a contradiction to the fact that we started with a non-zero element from $\CH_d(H_C')_{{\mathbb Q}}$. So we prove that the map from $\CH_d(H_C')_{{\mathbb Q}}$ to $\CH_d(J(C)')_{{\mathbb Q}}$ is injective.

Now $H_C'$ is birational to $H_C$ and $J(C)'$ is birational to $J(C)$. So we have the commutative diagram
$$
  \diagram
   \CH_d(H_C')_{{\mathbb Q}}\ar[dd]_-{} \ar[rr]^-{} & & \CH_d(J(C)')_{{\mathbb Q}} \ar[dd]^-{} \\ \\
  \CH_d(H_C)_{{\mathbb Q}} \ar[rr]^-{j_{C*}} & & \CH_d(J(C))_{{\mathbb Q}}
  \enddiagram
  $$
Then arguing as in proposition \ref{prop2} and noting that the support of a cycle on $H_C'$ does not lie on the exceptional locus of the birational map from $J(C)'$ to $J(C)$, we prove that the homomorphism $j_{C*}$ at the level of Chow groups of $k$-cycles with rational coefficients is injective.
\end{proof}

Now we want to prove the following.

\begin{proposition}
Let $H_C$ be as in the previous theorem. Then the push-forward $B_*(H_C)$ to $B_*(J(C))$ is injective, where $B_*$ denote the group of algebraic cycles modulo algebraic equivalence.
\end{proposition}

\begin{proof}

For that first we note that the Collino's argument as in \cite{Collino} goes through for $B_*$. Meaning that if we consider the closed embedding of $\Sym^m C$ into $\Sym^n C$, then the push-forward at the level of $B_*$ is injective. So let $\wt{C}$ be a curve which is a ramifield Galois cover of $C$, such that $J(C)$ is embedded in $J(\wt{C})$. Let $\theta_{\wt{C}}$ be the theta divisor of $J(\wt{C})$ and let $H_C=\theta_{\wt{C}}\cap J(C)$.
Let $H_C'$ be the inverse image of $H_C$ in $\Sym^{\wt{g-1}}(\wt{C})$ and $J(C)'$ be the inverse image of $\Sym^{\wt{g}}(\wt{C})$, where $\wt{g}$ is the genus of $\wt{C}$. First we prove that $B_*(H_C')$ to $B_*(J(C)')$ is injective. Then we consider the Cartesian square
$$
  \diagram
  H_C' \ar[dd]_-{} \ar[rr]^-{} & & J(C)' \ar[dd]^-{} \\ \\
  H_C \ar[rr]^-{} & & J(C)
  \enddiagram
  $$
argue as in proposition \ref{prop2} to get that $B_*(H_C)$ to $B_*(J(C))$ is injective. For that we prove the following lemma.

\begin{lemma}
\label{basechange}
Consider a Cartesian square
$$
  \diagram
   {Sym^n X}\times_Z Y\ar[dd]_-{} \ar[rr]^-{} & & Y \ar[dd]^-{} \\ \\
  \Sym^n X \ar[rr]^-{} & & Z
  \enddiagram
  $$
where $Y\to Z$ is an embedding then the inclusion $\Sym^m X\times_Z Y$ to $\Sym^n X\times_Z Y$. Let $j$ denote the embedding of $\Sym^m X\times_Z Y$ to $\Sym^n X\times_Z Y$ for $m\leq n$. Then $j_*$ is injective at the level of $B_*$.
\end{lemma}
\begin{proof}

Let $i$ be the embedding of $\Sym^m X\to \Sym^n X$, where $m\leq n$. Let $j$ denote the embedding of $\Sym^m X\times_Z Y\to \Sym^n X\times_Z Y$. Let $\Gamma$ be as before,
$$\Gamma=\pi_n\times \pi_n(Graph(pr_{n,m}))$$
where $pr_{n,m}$ is the projection from $X^n$ to $X^m$. $\pi_i$ is the natural morphism from $X^i$ to $\Sym^i X$. Let $\pi$ denote the projection morphism from $(\Sym^n X\times_Z Y)\times (\Sym^m X\times_Z Y)\to \Sym^n X\times \Sym^m X$. Then consider the correspondence
$$\pi^*(\Gamma)=\Gamma'$$
supported on $(\Sym^n X\times_Z Y)\times (\Sym^m X\times_Z Y)$. Arguing as in \ref{lemma1} in the previous section we can prove that $\Gamma'_*j_*$ is induced by $(j\times id)^*\Gamma'$, which is equal to $(j\times id)^*\pi^*\Gamma=(\pi\circ (j\times id))^*\Gamma$. Now we have the following commutative diagram.

$$
  \diagram
   {Sym^m X\times_Z Y}\times {\Sym^m X\times_Z Y}\ar[dd]_-{j\times id} \ar[rr]^-{\pi'} & & {\Sym^n X\times_Z Y}\times {\Sym^m X\times_Z Y} \ar[dd]^-{\pi} \\ \\
  \Sym^m X\times \Sym^m X \ar[rr]^-{i\times id} & & \Sym^n X\times \Sym^m X
  \enddiagram
  $$
So we get that
$$\pi\circ (j\times id)=(i\times id)\times \pi'$$
therefore we have that
$$(\pi\circ (j\times id))^*\Gamma=\pi'^*(i\times id)^*\Gamma\;.$$
Now
$$(i\times id)^*\Gamma=\Delta+Y_1$$
where $\Delta$ is the diagonal in $\Sym^m X\times \Sym^m X$ and $Y_1$ is supported on $\Sym^m X\times \Sym^{m-1}X$. Now we compute $\pi'^*(\Delta)$, that is
$$\{([x_1,\cdots,x_m],y)([x_1',\cdots,x_m'],y')|
[x_1,\cdots,x_m]=[x_1',\cdots,x_m']\}$$
but by the definition of fibered product we have that
$$f([x_1,\cdots,x_m])=g(y)=g(y')$$
assuming $g$ to be an embedding we get that $y=y'$. So
$$\pi'^*\Delta=\Delta_{\Sym^m X\times_Z Y}\;.$$
Now
$$\pi'^*(Y)=\{(([x_1,\cdots,x_m],y),([x_1',\cdots,x_m'],y'))\}$$
where $[x_1',\cdots,x_m']=[y_1,\cdots,y_{m-1},p]$, which means that
$$\pi'^* Y$$ is supported on
$$(\Sym^m X\times_Z Y)\times (\Sym^{m-1}X\times_Z Y)\;.$$
So
$$\pi'^*(\Delta+Y_1)=\Delta_{\Sym^m X\times_Z Y}+Y_2$$
where $Y_2$ is supported on
$$(\Sym^m X\times_Z Y)\times (\Sym^{m-1}X\times_Z Y)\;.$$
Now consider
$$\rho:\Sym^m X\times_Z Y\setminus \Sym^{m-1}X\times_Z Y\to \Sym^m X\times_Z Y;,$$
then
$$\rho^*\Gamma'_*j_*(V)=\rho^*(V\times \Sym^m X\times_Z Y.(\Delta_{\Sym^m X\times_Z Y}+Y_2))=\rho^*(V+V_1)=\rho^*(V)\;.$$
Here $V_1$ is supported on $\Sym^{m-1}X\times_Z Y$.

Now  to prove that $j_*$ is injective, we apply induction on $m$. If $m=0$, then since $Y\to Z$ is an embedding we have $\Sym^m \times_Z Y$ is a point. Since $\Sym^n X\times_Z Y$ is projective we have that the inclusion of the point into $\Sym^n X\times_Z Y$ induces injective $j_*$.

Consider the following commutative diagram,

$$
  \xymatrix{
   B_*(\Sym^{m-1}X\times_Z Y) \ar[r]^-{j'_{*}} \ar[dd]_-{}
  &   B_*({\Sym^m X\times_Z Y}) \ar[r]^-{\rho^{*}} \ar[dd]_-{j_{*}}
  & B_*(X_0(m))  \ar[dd]_-{}  \
  \\ \\
   B_*(\Sym^{m-1}X\times_Z Y) \ar[r]^-{j''_*}
    & B_*(\Sym^{n} X\times_Z Y) \ar[r]^-{}
  & B_*(U)
  }
$$
Here $X_0(m)$ is the complement of $\Sym^{m-1}X\times_Z Y$ in $\Sym^m X\times_Z Y$ and $U$ is the complement of $\Sym^{m-1}X\times_Z Y$ in $\Sym^n X\times_Z Y$.

Now suppose that
$$j_*(z)=0$$
that will imply that
$$\rho^*\Gamma'_*j_*(z)=0$$
that is
$$\rho^*(z)=0\;.$$
So by the exactness of the first row we get that
$$j'_*(z')=z\;.$$
Now we have that
$$j_*\circ j'_*(z')=j''_*(z')$$
but by the induction hypothesis we have $j''_*$ is injective from $\Sym^{m-1}X\times_Z Y$ to $\Sym^n X\times_Z Y$. Therefore by the commutativity we have that
$$j''_*(z')=0$$
hence $z'=0$ consequently $z=0$. So $j_*$ is injective.
\end{proof}

\begin{lemma}
Let $Y$ be a closed  subscheme of $\Sym^n X$. Let $i$ denote the closed embedding of $\Sym^m X$ into $\Sym^n X$. Consider $j:Y\cap \Sym^m X\to Y$. Then $j_*$ is injective at the level of $B_*$.
\end{lemma}

\begin{proof}
Follows from the previous proposition with $Z=\Sym^n X$.
\end{proof}
So we get that the push-forward homomorphism from $B_*(H_C')$ to $B_*(J(C)')$ is injective. Hence arguing as in \ref{prop2} we get that $B_*(H_C)$ to $B_*(J(C))$ is injective.

\end{proof}

\section{Collino's theorem for higher Chow groups}
\label{Collino}

Let $C$ be a smooth projective curve over  an algebraically closed field. Let $\Sym^n C$ denote the $n$-th symmetric power of $C$. Let us fix a point $p$ in $C$. Consider the closed embedding $i_{m,n}$ of $\Sym^m C$ to $\Sym^n C$, given by
$$[x_1,\cdots,x_m]\mapsto [x_1,\cdots,x_m,p,\cdots,p]$$
where $[x_1,\cdots,x_m]$ denote the unordered $m$-tuple of points in $\Sym^m C$. Then the push-forward homomorphism $i_{m,n*}$ from $\CH_*(\Sym^m C)$ to $\CH_*(\Sym^n C)$ is injective as proved in \cite[Theorem 1]{Collino}. In this section we prove that the same holds for the higher Chow groups. That is the push-forward homomorphism $i_{m,n*}^s$ from $\CH_*(\Sym^m C,s)$ to $\CH_*(\Sym^n C,s)$ is injective.  To prove that we follow the approach by Collino in \cite{Collino}, the argument present here is a minor modification of the arguments in \cite{Collino}, but we write it for our convenience.

Let $\Gamma^s$ be the correspondence given by
$$\pi_n\times \pi_m(\Gamma')$$
supported on $(\Sym^m C\times_{\Spec(k)} \Delta^s)\times_{\Spec(k)}(\Sym^n C\times_{\Spec(k)}\Delta^s)$ where $\Gamma'$ is the graph of the projection $pr_{n,m}^s$ from $( C^n\times_{\Spec(k)}\Delta^s)$ to $( C^m\times_{\Spec(k)} \Delta^s)$ and $\pi_{n}$ is the natural morphism from
$C^n\times_{\Spec(k)}\Delta^s$ to $\Sym^n C\times_{\Spec(k)}\Delta^s$. Let $g^s_*$ be the homomorphism induced by $\Gamma^s$ at the level of algebraic cycles.

First we prove the following lemma.
\begin{lemma}
\label{lemma3}
The homomorphism $g^s_{*}\circ i^s_{m,n*}$ at the level of the group of algebraic cycles, is induced by the cycle $(i_{m,n}^s\times id)^*\Gamma^s$ on $(\Sym^m C\times_{\Spec(k)} \Delta^s)\times (\Sym^m C\times_{\Spec(k)}\Delta^s)$.
\end{lemma}
\begin{proof}
Let's denote $i_{m,n*}^s$ as $i^s_*$.
We have
$$g^s_*i^s_*(Z)=pr_{(\Sym^m C\times\Delta^s)*}(i^s_*(Z)\times \Sym^m C\times \Delta^s.\Gamma^s)\;.$$
The above expression can be written as
$$pr_{(\Sym^m C\times\Delta^s)*}((i^s\times id)_*(Z\times \Sym^m C\times \Delta^s).\Gamma^s)\;.$$
By the projection formula the above is equal to
$$pr_{(\Sym^m C\times\Delta^s)*}\circ(i^s\times id)_*
((Z\times \Sym^m C\times \Delta^s). (i^s\times id)^*\Gamma^s)\;.$$
Since $pr_{\Sym^m C\times\Delta^s}\circ(i^s\times id)$
is the projection $pr_{\Sym^m C\times\Delta^s}$ we get that the above is equal to
$$pr_{(\Sym^m C\times\Delta^s)*}((Z\times \Sym^m C\times \Delta^s). (i^s\times id)^*\Gamma^s)\;.$$
Here the above two projections are taken respectively on $(\Sym^n C\times\Delta^s)\times(\Sym^m C\times\Delta^s)$ and on $(\Sym^m C\times\Delta^s)\times(\Sym^m C\times\Delta^s)$.
So we get that $g^s_*\circ i^s_*$ is induced by $(i^s\times id)^*\Gamma^s$.

\end{proof}
Now consider a closed subscheme $W$ of $\Sym^n C$. Let $i_{m,n}$ denote the embedding of $\Sym^m C$ into $\Sym^n C$. Consider the morphism $i_{m,n}^s$ from $(\Sym^m C\setminus i_{m,n}^{-1}W)\times \Delta^s$ to $(\Sym^n C\setminus W)\times \Delta^s.$ Consider the restriction of $\Gamma^s$ to $((\Sym^n C\setminus W)\times \Delta^s) \times ((\Sym^m C\setminus i_{m,n}^{-1}W)\times \Delta^s)$. Denote it by $\Gamma^{s}{'}$. Let $g^{s}{'}_*$ denote the homomorphism induced by $\Gamma^{s}{'}$. Then arguing as in the previous lemma \ref{lemma3} we get the following.

\begin{corollary}
The homomorphism $g^{s'}_*\circ i^s_{m,n*}$ is induced  by the cycle $(i^s_{m,n}\times id)^*\Gamma^{s}{'}$ on $((\Sym^m C \setminus i_{m,n}^{-1}W)\times \Delta^s)\times ((\Sym^m C \setminus i_{m,n}^{-1}W)\times \Delta^s)$.
\end{corollary}
\begin{proof}
It follows by arguing as in lemma \ref{lemma3} with $g^s_*,\Gamma^s$ replaced by $g^s{'}_*, \Gamma^s{'}$.
\end{proof}

Now let us consider the closed embedding $\Sym^{m-1}C\times \Delta^s$ into $\Sym^m C\times \Delta^s$, induced by the embedding $\Sym^{m-1}C$ into $\Sym^m C$. Let $\rho^s$ be the embedding of the complement of $\Sym^{m-1}C\times \Delta^s$ in $\Sym^m C\times \Delta^s$. Then we have the following proposition.
\begin{proposition}
\label{prop3}
At the level of the group of algebraic cycles we have

$$\rho^{s*}\circ g^s_*\circ i^s_*=\rho^{s*}\;.$$
\end{proposition}
\begin{proof}
To prove the proposition we prove that
$$(i^s\times \id)^{-1}\Gamma^s=\Delta \cup D$$
where $\Delta $ means the diagonal in $(\Sym^{m}C\times\Delta^s)\times (\Sym^{m}C\times\Delta^s)$ and $D$ is a closed subscheme of $(\Sym^{m}C\times\Delta^s)\times (\Sym^{m-1}C\times\Delta^s)$.
For that we write out
$$(i^s\times \id)^{-1}\Gamma^s\;,$$
that is equal to
$$(i^s\times \id)^{-1}(\pi_n\times\pi_m)Graph(pr^s_{n,m})\;.$$
The above is equal to
$$(i^s\times \id)^{-1}(\pi_n\times\pi_m)
\{((x_1\cdots,x_n,\delta_s),(x_1,\cdots,x_m,\delta^s))|x_i\in C, \delta^s\in \Delta^s\}$$
that is
$$(i^s\times \id)^{-1}\{([x_1,\cdots,x_n,\delta^s],[x_1,\cdots,x_m,\delta^s])|x_i\in C,\delta^s\in \Delta^s\}\;.$$
Call the set
$$\{([x_1,\cdots,x_n,\delta^s],[x_1,\cdots,x_m,\delta^s])|x_i\in C,\delta^s\in \Delta^s\}$$
as $B$, and the set
$$(i^s\times \id)^{-1}\{([x_1,\cdots,x_n,\delta^s],[x_1,\cdots,x_m,\delta^s])|x_i\in C,\delta^s\in \Delta^s\}\;.$$
as $A$. The set $A$ is of the form
$$\{([x'_1,\cdots,x'_m,\delta^s],[y'_1,\cdots,y'_m,\delta^s])|
([x_1',\cdots,x_m',p,\cdots,p,\delta^s],[y_1',\cdots,y_m',\delta^s])\in B\}\;.$$
So the set $A$ can be written as the union of
$$\{([x_1'\cdots,x_m',\delta^s],[x_1'\cdots,x_m',\delta^s])|x_i\in C,\delta^s\in \Delta^s\}$$
and
$$\{([x_1'\cdots,x_m',\delta^s],[x_1'\cdots,p,x_m',\delta^s])|x_i\in C,\delta^s\in \Delta^s\}\;,$$
that is the union
$$\Delta\cup D$$
where $\Delta $ is the diagonal in the scheme $(\Sym^{m}C\times\Delta^s)\times (\Sym^{m}C\times\Delta^s)$ and $D$ is a closed subscheme in $(\Sym^m C\times\Delta^s)\times (\Sym^{m-1}C\times\Delta^s)\;.$
Therefore we get that
$$(i^s\times id)^*(\Gamma)=\Delta+Y$$
where $Y$ is supported on $(\Sym^m\times\Delta^s)\times (\Sym^{m-1}C\times\Delta^s)$.
So $g_*i^s_*(Z)$ is equal to
$$\pr_{\Sym^m C\times \Delta^s*}[(\Delta+Y).(Z\times \Sym^m C\times \Delta^s)]=Z+Z_1$$
where $Z_1$ is supported on $\Sym^{m-1}C\times\Delta^s$. So
$$\rho^{s*}g_*i^s_*=\rho^{s*}(Z+Z_1)=\rho^{s*}(Z)$$
since $\rho^{s*}(Z_1)=0$. Hence the proposition is proved.
\end{proof}
Now we want to run the same argument as in proposition \ref{prop3} but for open varieties. That is let $W$ be a closed subscheme in $\Sym^n C$.
Let us consider the  embedding of $(\Sym^{m-1}C\setminus i_{m-1,n}^{-1}W)\times \Delta^s$ into $(\Sym^m C\setminus i_{m,n}^{-1}W)\times \Delta^s$, induced by the embedding $\Sym^{m-1}C$ into $\Sym^m C$. Let $\rho^s{'}$ be the embedding of the complement of $(\Sym^{m-1}C \setminus i_{m-1,n}^{-1}W)\times \Delta^s$ in $(\Sym^m C \setminus i_{m,n}^{-1}W)\times \Delta^s$. Then arguing as in proposition \ref{prop3} we prove that
\begin{corollary}
\label{corollary1}
At the level of algebraic cycles we have
$$
\rho{^s{'}}{^*}\circ g^s{'}_*\circ i_{m,n*}^s=\rho{^s{'}}{^*}\;.
$$
\end{corollary}
\begin{proof}
We argue as in proposition \ref{prop4} with $g^s_{*}$ replaced by $g^s{'}_*$ and $\Gamma^s$ by $\Gamma^{s}{'}$ and noting that
$$(i_{m,n}^s\times id)^*(\Gamma^{s}{'})=((i_{m,n}^s\times id)^*\Gamma^{s})\cap ((\Sym^m C\setminus i_{m,n}^{-1}W\times\Delta^s)\times (\Sym^m C\setminus i_{m,n}^{-1}W\times \Delta^s))\;.$$
\end{proof}

Now we prove that the push-forward homomorphism $i^s_*$ from $\CH_*(\Sym^m C,s)$ to $CH_*(\Sym^n C,s)$ is injective. This involves many steps. The first step would be to verify that the push-forward homomorphism $i^s_*$ is defined at the level of higher Chow groups.
Here $\bcZ$ denotes the group of admissible cycles, as defined by S. Bloch \cite{Bloch}.

\begin{lemma}
\label{lemma1}
${i_{m,n*}^s}$ is well defined from $\CH^*(\Sym^m C,s)$ to $\CH^*(\Sym^n C,s)$.
\end{lemma}
\begin{proof}
The morphism $i_{m,n}$ is defined from $\Sym^m C$ to $\Sym^n C$. That will give us a morphism  $i^s_{m,n}$ from $\Sym^m C\times \Delta^s$ to $\Sym^n C\times \Delta^s$. So consider the face morphisms
$$\partial_i:\Delta^{s-1}\to \Delta^s$$
given by
$$(t_0,\cdots,t_{s-1})\mapsto (t_0,\cdots,t_{i-1},0,t_i,\cdots,t_{s-1})\;.$$
This face morphisms give rise to the morphisms from $\Sym^m C\times \Delta^{s-1}$ to $\Sym^m C\times \Delta^s$, continue to call these morphisms as $\partial_i$.
Consider the following commutative diagram
$$
  \diagram
   \Sym^m C\times \Delta^{s-1}\ar[dd]_-{i_{m,n}^{s-1}} \ar[rr]^-{\partial_i} & & \Sym^{m}C\times \Delta^s \ar[dd]^-{i_{m,n}^s} \\ \\
  \Sym^n C\times \Delta^{s-1}\ar[rr]^-{\partial_i} & & \Sym^n C\times \Delta^s
  \enddiagram
  $$
From the above commutative diagram we get the commutativity of the following diagram:
$$
  \diagram
   \bcZ^*(\Sym^m C\times \Delta^{s})\ar[dd]_-{i_{m,n*}^{s}} \ar[rr]^-{\partial_i^*} & & \bcZ^*(\Sym^{m}C\times \Delta^{s-1}) \ar[dd]^-{i_{m,n*}^{s-1}} \\ \\
  \bcZ^*(\Sym^n C\times \Delta^{s})\ar[rr]^-{\partial_i^*} & & \bcZ^*(\Sym^n C\times \Delta^{s-1})
  \enddiagram
  $$
The commutativity  of this diagram and induced maps on admissible cycles shows that $i_{m,n*}^s$ is well defined at the level of higher Chow groups.
\end{proof}

\begin{corollary}
Let $W$ be a closed subscheme of $\Sym^n C$. Consider the morphism $i$ from $\Sym^m C\setminus i^{-1}(W)$ to $\Sym^n C\setminus W$. Then the homomorphism $i_{m,n*}^s$ is well defined from $\CH^*(\Sym^m C\setminus i^{-1}(W),s)$ to $\CH^*(\Sym^n C\setminus W,s).$
\end{corollary}
\begin{proof}
Proof follows by arguing similarly as in lemma \ref{lemma1} with $\Sym^m C,\Sym^n C$ replaced by $\Sym^m C\setminus i^{-1}(W),\Sym^n C\setminus W$.
\end{proof}

\begin{lemma}
\label{lemma2}
Let $\rho$ be the inclusion from $\Sym^m C\setminus \Sym^{m-1}C=C_0(m)$ to $\Sym^m C$. Then the homomorphism $\rho^{s*}$ is well defined at the level of higher Chow groups.
\end{lemma}
\begin{proof}
To prove this consider the diagram at the level of schemes.
$$
  \diagram
   C_0(m)\times \Delta^{s-1}\ar[dd]_-{\partial_i} \ar[rr]^-{\rho^{s-1}} & & \Sym^{m}C\times \Delta^{s-1} \ar[dd]^-{\partial_i} \\ \\
   C_0(m)\times \Delta^{s}\ar[rr]^-{\rho^s} & & \Sym^m C\times \Delta^s
  \enddiagram
  $$
That gives the following commutative diagram at the level of $\bcZ^*$.
$$
  \diagram
   \bcZ^*(\Sym^m C\times \Delta^{s})\ar[dd]_-{\partial_i^*} \ar[rr]^-{\rho^{s*}} & & \bcZ^*(C_0(m)\times \Delta^{s-1}) \ar[dd]^-{\partial_i^*} \\ \\
   \bcZ^*(\Sym^m C\times \Delta^{s-1})\ar[rr]^-{\rho^{s-1*}} & & \bcZ^*(C_0(m)\times \Delta^{s-1})
  \enddiagram
  $$
Therefore we have $\rho^{s*}$ is well defined at the level of higher Chow groups.
\end{proof}

\begin{corollary}
Let $W$ be a closed subscheme in $\Sym^m C$. Denote the complement of $\Sym^{m-1} C\setminus i_{m,n}^{-1}W$ in $\Sym^m C\setminus W$ as $W_0(m)$. Let $\rho$ be the inclusion of $W_0(m)$ into $\Sym^m C\setminus W$. Then the homomorphism $\rho^{s*}$ is well defined from $\CH^*(\Sym^m C\setminus W,s)$ to $\CH^*(W_0(m),s)$.
\end{corollary}
\begin{proof}
Proof follows by arguing as in lemma \ref{lemma2} with $C_0(m),\Sym^m C$ replaced by $W_0(m),\Sym^m C\setminus W$.
\end{proof}

\begin{proposition}
\label{prop4}
The push-forward homomorphism $i^s_*$ from $\CH^*(\Sym^m X,s)$ to $\CH^*(\Sym^n X,s)$ is injective.
\end{proposition}
\begin{proof}
We prove this by induction. First $\Sym^0 C$ is a single point and the morphism $i^s_{0,n}=(p,\cdots,p)$, so the push-forward induced by this morphism is injective. Assume now that $i^s_*$ is injective for $m-1$ and any $n$ greater than or equal to $m-1$. Then consider the following commutative diagram

$$
  \xymatrix{
  0 \ar[r]^-{}  & \CH^*(\Sym^{m-1} C,s) \ar[r]^-{i^s_{m-1,m*}} \ar[dd]_-{}
  &   \CH^*(\Sym^m C,s) \ar[r]^-{\rho^{s*}} \ar[dd]_-{i^s_{mn*}}
  & \CH^*(C_0(m),s)  \ar[dd]_-{}  \
  \\ \\
  0 \ar[r]^-{} & \CH^*(\Sym^{m-1} C,s) \ar[r]^-{i^s_{m-1,n*}}
    & \CH^*(\Sym^n C,s) \ar[r]^-{}
  & \CH^*({(\Sym^{m-1}C)^c,s})
  }
$$
In the above $(\Sym^{m-1}C)^c$ is the complement of $\Sym^{m-1}C$ in $\Sym^n C$.
In this diagram the left part of the two rows are exact by the induction hypothesis and the middle part is exact by the localization exact sequence for higher Chow groups.
Now suppose that $z$ belongs to $\CH^*(\Sym^m C,s)$, such that $$i^s_{m,n*}(z)=0$$
and let $Z$ be the cycle such that the cycle class of $Z$ is $z$. Let $cl(Z)$ denote the cycle class in the Higher Chow group, corresponding to the algebraic cycle $Z$.

Then we have
$$cl(\rho^{s*}g^s_*i^s_*(Z))=0$$
which means by the proposition \ref{prop3}
$$cl(\rho^{s*}(Z))=0\;,$$
hence
$$\rho^{s*}(cl(Z))=\rho^{s*}(z)=0\;.$$
$$$$
So by the localization exact sequence there exists $z'$ in $\CH^*(\Sym^{m-1}C,s)$, such that
$$z=i^s_{m-1,m*}(z')\;.$$
By the commutativity of the left square of the above commutative diagram we get that
$$i^s_{m-1,n*}(z')=0\;.$$
By the injectivity of $i^s_{m-1,n*}$ we get that $z'=0$, so $z=0$, hence $i^s_{m,n*}$ is injective.
\end{proof}

\begin{corollary}\label{openCollino}
Let $W$ be a closed subscheme inside $\Sym^n C$. Consider the embedding $i_{m,n}$ from $\Sym^m C\setminus i_{m,n}^{-1}(W)$ to $\Sym^n C\setminus W$. Then the homomorphism $i_{m,n*}^s$ from $\CH_*(\Sym^m C\setminus i_{m,n}^{-1}(W),s)$ to $\CH_*(\Sym^n C\setminus W,s)$ is injective.
\end{corollary}
\begin{proof}
The proof follows by arguing as in proposition \ref{prop4} with $\Sym^m C,\Sym^n C$ replaced by $\Sym^m C\setminus i_{m,n}^{-1}(W), \Sym^n C\setminus W$ respectively and by corollary \ref{corollary1}.
\end{proof}


\end{document}